\documentclass[journal]{IEEEtran}
\usepackage[utf8]{inputenc}
\usepackage{mathtools}
\mathtoolsset{showonlyrefs,showmanualtags}
\usepackage{amssymb}
\usepackage{dsfont}
\usepackage{booktabs}
\usepackage{cite}
\usepackage[final]{graphicx}
\usepackage{color}

\usepackage{nth}

\usepackage{amssymb,amsthm}

\newtheorem{Assumption}{Assumption}
\newtheorem{Theorem}{Theorem}

\newtheorem{Remark}{Remark}

\newtheorem{Corollary}{Corollary}

\newtheorem{Problem}{Problem}

\makeatletter
\newcommand*{\rom}[1]{\expandafter\@slowromancap\romannumeral #1@}
\makeatother

\begin{document}

\title{Optimal Distributed  Control  for  Leader-Follower Networks: A Scalable Design}
\author{Jalal Arabneydi, Mohammad M. Baharloo, and Amir G. Aghdam
\thanks{ This work has been supported in part by the Natural Sciences and Engineering Research Council of Canada (NSERC) under Grant RGPIN-262127-17, and in part by Concordia University under Horizon Postdoctoral Fellowship.}  
\thanks{Jalal Arabneydi, Mohammad M. Baharloo, and Amir G. Aghdam are with the  Department of Electrical and Computer Engineering, 
        Concordia University, 1455 de Maisonneuve Blvd. West, Montreal, QC, Canada, Postal Code: H3G 1M8.  Email:{\small jalal.arabneydi@mail.mcgill.ca},
         {\small baharloo@ieee.org}, and 
        {\small aghdam@ece.concordia.ca.}}%
}

\markboth{}%
{Arabneydi \MakeLowercase{\textit{et al.}}: }

\maketitle

\vspace*{-5.2cm}{\footnotesize{Proceedings of IEEE Canadian Conf. on Elec. and Comp. Eng., 2018.}}
\vspace*{4.45cm}

 \begin{abstract}
The focus of this paper is  directed towards optimal control  of multi-agent  systems consisting of one leader and a  number of followers in the presence of noise.  The dynamics  of every agent is assumed to be  linear,  and the performance index is a quadratic function of the states and actions of the leader and followers. The leader and followers  are coupled in both dynamics and cost.  The state of the leader and the average of the states of all  followers (called mean-field)  are common information and  known to all agents; however,  the local state of  the followers are private information and unknown to  other agents.
It is shown that the optimal distributed control  strategy is linear time-varying,    and its computational complexity is independent of the number of followers.   This strategy can be computed in a distributed manner, where the leader needs to solve one Riccati equation to determine  its optimal strategy while  each follower  needs to solve  two Riccati equations to obtain its optimal strategy. 
  This result is  subsequently extended to the case of the  infinite horizon  discounted and undiscounted cost functions, where the optimal distributed strategy is shown to be stationary.  A numerical example with $100$ followers is provided to demonstrate the efficacy of the results.
\end{abstract}

\section{Introduction}
There has been a growing interest in networked control systems in recent years, due to their applications in emerging areas such as control of a platoon of autonomous vehicles, environmental monitoring using sensor networks, and surveillance using a team of UAVs \cite{Fax2004,rawat2014wireless,UAVsurvey2016}. In particular, in the control of multi-agent systems, every agent exchanges some information with a subset of agents in order to properly coordinate its position and movement such that a global objective is achieved. To this end, each agent requires some computational effort in order to compute its control action based on the information available to it.

The leader-follower structure is particularly very common in the coordination control of multi-agent networks. In this type of system, each agent is either a leader or a follower, where the movement of the followers is dependent on the trajectory of the leader(s).  A global objective that is of special interest is consensus, where the states of the agents are desired to converge to a common value~\cite{Jadbabaie2003,Olfati2007survey,Ren2007survey}. If the communication graph of the network is connected, then  consensus can be reached using a linear strategy. However, such strategy suffers from the curse of dimensionality, in general. In addition, the required amount of communication between agents under this type of strategy is typically high. This can lead to some practical problems, specially given that the battery consumption of each node is closely related to the amount of its required communication.  An optimal control strategy can therefore be very important for the efficient use of resources in the network. Inspired by this objective,  the authors in~\cite{Movric2014,Cao2010} study distributed linear quadratic control. Although the above techniques are effective in many cooperative control applications, they are limited by the computational cost, making them unsuitable for large-scale networks.  A large-scale network of homogeneous agents with decoupled dynamics is investigated in~\cite{Borelli2008}, for which the infinite-horizon optimal control strategy is obtained by solving two scalable coupled algebraic Riccati equations.

The present work aims to address the above shortcomings for  a leader-follower multi-agent network with a large number of followers. It is assumed that the graph representing the network is such that the average of the states of the followers and the local state of the leader are available to every agent.  In contrast to~\cite{Borelli2008},  the network considered in this paper has a leader,   the dynamics of agents are coupled, and the cost function can be  either finite-horizon or infinite-horizon.   The mean-field team approach~\cite{arabneydi2016new} is used in this paper to obtain the optimal control strategy by solving two scalable decoupled Riccati equations.

This paper is organized as follows. In Section~\ref{sec:prob}, a leader-follower multi-agent network with the mean-field information structure is formulated. The main result of the paper is developed in Section~\ref{sec:main} for the finite-horizon cost function, and is then extended to the infinite-horizon case in Section~\ref{sec:inf}.  Simulation results are provided in Section~\ref{sec:num}, and the paper is concluded in Section~\ref{sec:conc}.

\section{Problem Formulation}\label{sec:prob}

In this paper,  $\mathbb{N}$ and $\mathbb{R}$ represent natural and real numbers, respectively, and  given any $k \in  \mathbb{N}$,  the finite set of integers   $\{1,2, \ldots,k\}$  is denoted by $\mathbb{N}_k$.

Consider  a multi-agent network with one leader and  $n \in \mathbb{N}$ followers,  operating over a finite control horizon $T \in \mathbb{N}$.  Let  $x^i_t \in \mathbb{R}^{d_x}$ and $ u^i_t \in \mathbb{R}^{d_u},$  $d_x,d_u \in \mathbb{N}$, denote the state and action of follower $ i \in \mathbb{N}_n$  at time $t \in \mathbb{N}_T$, respectively.   Denote  the average  of the states of the followers  at time~$t$~by
\[\bar x_t=\frac{1}{n}\sum_{i=1}^n x^i_t.\]
Following the terminology of mean-field teams~\cite{arabneydi2016new}, we refer to the average of the states of the followers as \emph{mean-field} in the sequel.  Let  also $x^0_t \in \mathbb{R}^{d_x}$ and $ u^0_t \in \mathbb{R}^{d_u}$ denote the state and  action of the leader at time $t \in \mathbb{N}_T$, respectively.  The dynamics of the leader at time $\mathbb{N}_T$  is given~by
\begin{equation}\label{eq:dynamics-leader}
x^0_{t+1}=A^0_tx^0_t+B^0_tu^0_t+D^0_t \bar x_t+ w^0_t,
\end{equation}
where $w^0_t \in \mathbb{R}^{d_x}$ is the state noise of the leader. Similarly, the dynamics of  follower $i \in \mathbb{N}_n$ is described by
\begin{equation}\label{eq:dynamics-followers}
x^i_{t+1}=A_tx^i_t+B_tu^i_t+D_t \bar x_t+E_tx^0_t+w^i_t,
\end{equation}
where $w^i_t \in \mathbb{R}^{d_x}$ is the state noise of the $i$-th follower.  In  general,  the leader's dynamics  may depend on the states of the followers. Similarly,  the followers' dynamics  may depend on  the state of the leader as well as the states  of followers. 

At each time  $t \in \mathbb{N}_T$, the leader observes  its local state and the mean-field, i.e., 
\begin{equation}\label{eq:info-leader}
u^0_t=g^0_t(x^0_t,\bar x_t),
\end{equation}
where $g^0_t: (\mathbb{R}^{d_x})^2 \rightarrow \mathbb{R}^{d_u}$. Furthermore, each follower $i \in \mathbb{N}_n$ observes  its local state, the state of the leader, and  the mean-field,  i.e.,
\begin{equation}\label{eq:info-follower}
u^i_t=g^i_t(x^i_t, \bar x_t,x^0_t),
\end{equation}
where $g^i_t:(\mathbb{R}^{d_x})^3  \rightarrow \mathbb{R}^{d_u}$.  Under this information structure,  the privacy of each follower is preserved, i.e.,  the local state of each follower is only known to itself. Note that there are different ways to share the mean-field $ \bar x_t$ among the agents, depending on the structure of the communication graph. For example,  all agents can send their states to the leader and then the  leader computes the mean and sends it back to every follower (in which case, a link is required between the leader and every follower). Alternatively, each agent can run a consensus algorithm to compute the mean-field within the control time interval.\footnote{In practical applications, the control operation has a much longer time-scale compared to   the communication operation.
}


It is desired  that the leader and followers minimize a prescribed quadratic cost function, while achieving a global objective (such as consensus) as a group. This cost function can, for instance, reflect  the coordination error of the agents as well as the energy consumption of the actuators, collectively. To this end, consider the following optimization problem.
\begin{Problem}\label{prob}
Given the dynamics~\eqref{eq:dynamics-leader} and~\eqref{eq:dynamics-followers} as well as the information structures~\eqref{eq:info-leader} and~\eqref{eq:info-follower},  find  the optimal strategy that  minimizes the following performance index:
 \begin{multline}\label{eq:cost-function}
J_T=\mathbb{E} \Big[\sum_{t=1}^T  (x^0_t)^\intercal Q^0_t x^0_t + (u^0_t)^\intercal R^0_t u^0_t\\
+\frac{1}{n} \sum_{i=1}^n  (x^i_t)^\intercal Q_t x^i_t+( x^i_t-x^0_t)^\intercal P_t (x^i_t - x^0_t)+ (u^i_t)^\intercal R_t u^i_t\\
+\frac{1}{2n^2} \sum_{i=1}^n   \sum_{j=1}^n  ( x^i_t-x^j_t)^\intercal H_t (x^i_t - x^j_t) \Big],
\end{multline}
where matrices $Q^0_t, R^0_t,Q_t, P_t, R_t,$ and $H_t$ are symmetric.
\end{Problem}
Note that the rate of convergence of the followers to the leader is directly dependent on  matrix  $P_t$  in~\eqref{eq:cost-function}. Similarly, the movement of  the followers  as a group depends on matrix $H_t$.
\begin{Remark}
\emph{ In the special case when matrices $B^0_t$, $D^0_t$, $Q^0_t,$ and $R^0_t$ are zero, Problem~\ref{prob} becomes  the optimal control of a leaderless multi-agent system, where it is desired that the followers  track the reference signal~$x^0_t$.}
\end{Remark}


In general, Problem~\ref{prob} is difficult to solve due to its complex information structure.  Since  neither  the information structure is partially nested~\cite{ho1972team} nor the problem is  quadratic invariant~\cite{rotkowitz2006characterization},  and in addition  the noise processes are not  necessarily Gaussian, one can not assume that the optimal strategy is linear.\footnote{When the information structure is non-classical, the optimal strategy  may  not be linear~\cite{Witsenhausen1968Counterexample}.}  Moreover,  the  dimension of the augmented matrices, which are fully dense,  increases with the number of followers, i.e.,   solving Problem~\ref{prob} using existing techniques can be computationally expensive for a large number of followers.

\section{Main Result}\label{sec:main}
In this section,   the main result of this paper is presented.  It is assumed that the primitive random variables  satisfy the following standard assumption.
\begin{Assumption}\label{assump:noise}
The initial states and noise processes are mutually independent in time.
\end{Assumption}
At  any time $t \in \mathbb{N}_T$, define the following matrices:
\begin{equation}
\bar A_t:=\left[ \begin{array}{cc}
A^0_t &D^0_t\\
E_t &A_t+D_t
\end{array}  \right],  \bar B_t:=\left[ \begin{array}{cc}
B^0_t &0_{d_{x} \times d_{u}}\\
0_{d_{x} \times d_{u}} &B_t
\end{array}  \right],
\end{equation}
\begin{equation}
\bar Q_t:=\left[ \begin{array}{cc}
Q^0_t+P_t & -P_t\\
-P_t & Q_t+P_t
\end{array}  \right],  \bar R_t:=\left[ \begin{array}{cc}
R^0_t &0_{d_{u} \times d_{u}}\\
0_{d_{u} \times d_{u}} &R_t
\end{array}  \right].
\end{equation}
\begin{Assumption}\label{assump:positive}
Matrices  $Q_t +P_t +H_t$ and $\bar Q_t$  are positive semi-definite and  matrices $R^0_t$ and $ R_t$ are positive definite.
\end{Assumption}
For any $t \in \mathbb{N}_T$, define the following Riccati equation:
\begin{multline}\label{eq:R1}
\breve M_t = - A_t^\intercal   \breve M_{t+1} B_t  \left( B_t^\intercal  \breve M_{t+1}  B_t  +  R_t  \right)^{-1}  B_t^\intercal   \breve M_{t+1}  A_t\\
 + A_t^\intercal  \breve M_{t+1} A_t + Q_t+P_t+H_t,
\end{multline}
where  $\breve M_{T+1}=0_{d_{x} \times d_{x}}$. Define also the following Riccati equation: 
\begin{multline}\label{eq:R2}
\bar M_t = - \bar A_t^\intercal   \bar M_{t+1} \bar B_t  \left( \bar B_t^\intercal  \bar M_{t+1} \bar  B_t  +  \bar R_t  \right)^{-1}  \bar B_t^\intercal   \bar M_{t+1}  \bar A_t\\
 +\bar  A_t^\intercal  \bar  M_{t+1} \bar A_t +\bar Q_t,
\end{multline}
for  any  $t \in \mathbb{N}_T$, with  $\bar M_{T+1}= \left[ \begin{array}{cc}
0_{d_{x} \times d_{x}} &0_{d_{x} \times d_{x}}\\
0_{d_{x} \times d_{x}} &0_{d_{x} \times d_{x}}
\end{array}  \right]$.
\begin{Theorem}\label{thm:finite}
Let Assumptions~\ref{assump:noise} and~\ref{assump:positive} hold. The optimal strategy for Problem~\ref{prob} is linear and unique, i.e.,
  \begin{align}\label{optimal_action_leader}
    {u^0_t}&=\bar L^{1,1}_t x^0_t+  \bar L^{1,2}_t \bar x_t, \nonumber \\
    {u^i_t}&=\breve L_tx^i_t + \bar L_t^{2,1} x^0_t+ (\bar L_t^{2,2}-\breve L_t)\bar x_t,
  \end{align}
  where the gains $\{ \breve L_t, \bar L_t\}_{t=1}^{T-1}$ are obtained by solving Riccati equations~\eqref{eq:R1} and~\eqref{eq:R2} as follows: 
\begin{align}
\breve L_t&=-\left( B_t^\intercal  \breve M_{t+1}  B_t  +  R_t  \right)^{-1}  B_t^\intercal   \breve M_{t+1}  A_t,\nonumber \\
\bar L_t&=\left[ 
\begin{array}{cc}
\bar L^{1,1}_t & \bar L^{1,2}_t\\
\bar L^{2,1}_t & \bar L^{2,2}_t
\end{array}
\right]
=-\left( \bar B_t^\intercal  \bar M_{t+1}  \bar B_t  +  \bar R_t  \right)^{-1}  \bar B_t^\intercal   \bar M_{t+1}  \bar A_t.
\end{align}
\end{Theorem}
\begin{proof}
Suppose every agent knows the centralized information $\{x^0_t,x^1_t,\ldots, x^n_t\}$. Let $\bar u_t=\frac{1}{n} \sum_{i=1}^n u^i_t$ and $\bar w_t=\frac{1}{n} \sum_{i=1}^n w^i_t$. We first transform the problem by using an isomorphic transformation and  solve the transformed problem.  Then, we  use the inverse transformation to transform  the obtained solution to the solution of the original problem,  and show that the resultant (centralized) solution  is implementable under the decentralized information structure.  Define $\breve x^i_t:= x^i_t- \bar x_t$,  $\breve u^i_t:= u^i_t- \bar u_t$, and $\breve w^i_t:= w^i_t- \bar w_t$. From~\eqref{eq:dynamics-leader} and~\eqref{eq:dynamics-followers}, 
\begin{align}\label{eq:dynamics-breve}
\breve x^i_{t+1}&=A_t \breve x^i_t+ B_t \breve u^i_t+\breve w^i_t,  \nonumber \\
 \left[\begin{array}{c}
x^0_{t+1}\\
\bar x_{t+1}
\end{array}\right]&=\bar A_t  \left[\begin{array}{c}
x^0_t\\
\bar x_t
\end{array}\right]+\bar B_t  \left[\begin{array}{c}
u^0_t\\
\bar u_t
\end{array}\right]+  \left[\begin{array}{c}
w^0_t\\
\bar w_t
\end{array}\right].
\end{align}
Rewrite the cost function $J_T$, given by ~\eqref{eq:cost-function},  in terms of the new variables as follows:
\begin{multline}\label{eq:proof_cost-1}
 \hspace{-.4cm} J_T\hspace{-.1cm}=\hspace{-.1cm}\mathbb{E} \Big[ \hspace{-.05cm}\sum_{t=1}^T  (x^0_t)^\intercal Q^0_t x^0_t + (u^0_t)^\intercal R^0_t u^0_t
+\frac{1}{n}  \hspace{-.1cm} \sum_{i=1}^n  (\breve x^i_t + \bar x_t)^\intercal Q_t (\breve x^i_t+ \bar x_t) \\
+\frac{1}{n}\sum_{i=1}^n (\breve  x^i_t + \bar x_t-x^0_t)^\intercal P_t (\breve x^i_t+ \bar x_t - x^0_t)+ (\breve u^i_t+\bar u_t)^\intercal R_t (\breve u^i_t+\bar u_t)\\
+\frac{1}{2n^2} \sum_{i=1}^n   \sum_{j=1}^n  ( \breve x^i_t-\breve x^j_t)^\intercal H_t (\breve x^i_t - \breve x^j_t) \Big].
\end{multline}
The above  relation can  be simplified by exploiting the fact that $\frac{1}{n}\sum_{i=1}^n \breve{x}^i_t=0_{d_x \times 1}$ and $\frac{1}{n}\sum_{i=1}^n \breve{u}^i_t=0_{d_u \times 1}$. This leads to the following simplified equation:
\begin{align}\label{eq:proof-cost}
J_T=\mathbb{E} \Big[\sum_{t=1}^T  \left[\begin{array}{c}
x^0_t\\
\bar x_t
\end{array}\right]^\intercal \bar Q_t \left[\begin{array}{c}
x^0_t\\
\bar x_t
\end{array}\right]+  \left[\begin{array}{c}
u^0_t\\
\bar u_t
\end{array}\right]^\intercal \bar R_t  \left[\begin{array}{c}
u^0_t\\
\bar u_t
\end{array}\right] \nonumber \\
+\frac{1}{n}\sum_{i=1}^n (\breve  x^i_t)^\intercal (Q_t+P_t+H_t) (\breve x^i_t)+ (\breve u^i_t)^\intercal R_t (\breve u^i_t) \Big].
\end{align}
The cost function $J_T$ in~\eqref{eq:proof-cost} is  the  sum of the cost functions of~$n$ systems with state and action $(\breve x^i_t, \breve u^i_t)$, $i \in \mathbb{N}_n,$ and  one system with the state and action $(\{x^0_t, \bar x_t\}, \{u^0_t, \bar u_t\})$. These $n+1$ systems are decoupled  due to the certainty equivalence theorem~\cite{caines1987linear}. As a result,   $J_T$  is minimized when  the cost functions of the  $n+1$  systems with decoupled dynamics and cost are minimized, i.e., 
\begin{equation}
\left[\begin{array}{c}
u^0_t\\
\bar u_t
\end{array}\right]= \bar L_t \left[\begin{array}{c}
x^0_t\\
\bar x_t
\end{array}\right] \quad \text{and}\quad \breve u^i_t=\breve L_t \breve x^i_t, \quad i \in \mathbb{N}_n.
\end{equation}
We  now transform the above solution to the solution of the  original problem: 
\begin{equation}\label{eq:obtained-solution}
u^i_t=\breve u^i_t+\bar u_t=\breve L_t (x^i_t - \bar x_t)+ [\bar L^{2,1}_t \quad \bar L^{2,2}_t] \left[\begin{array}{c}
x^0_t\\
\bar x_t
\end{array}\right].
\end{equation}
The transformed solution is optimal for the original decentralized problem because it is  implementable under the  information structure in Section~\ref{sec:prob}.
\end{proof}
According to Theorem~\ref{thm:finite}, the leader must solve the Riccati equation~\eqref{eq:R2} to determine  its optimal strategy ($\bar L^{1,1}_t$, $\bar L^{1,2}_t$) whereas  each follower must solve   the Riccati equations~\eqref{eq:R1} and \eqref{eq:R2} to find its optimal strategy ($\breve L_t, \bar L^{2,1}_t,$ $\bar L^{2,2}_t$): one for the local adjustment with their average (i.e., mean-field) and one for the global adjustment with the leader.
\begin{Corollary}\label{cor:finite}
Let matrices $\breve  L_t$ and $\bar L^{1,1}_t$ in Theorem~\ref{thm:finite} be invertible. The optimal strategy  can  be rewritten in the form of the solution of a standard consensus problem as follows:
  \begin{align}
    {u^0_t}&=  \sum_{i=1}^n \alpha_t(x^0_t- \beta_t x^i_t), \nonumber \\
    {u^i_t}&= \sum_{j=1}^n  \gamma_t(x^i_t- \mu_t x^j_t) + \sum_{i=1}^n   \lambda_t (x^0_t- x^i_t),
  \end{align}
  where $\alpha_t=\frac{1}{n} \bar L^{1,1}_t$, $\beta_t=-(\bar{L}^{1,1}_t)^{-1} \bar L^{1,2}_t, \gamma_t=\frac{1}{n}\breve L_t, \mu_t=-(\breve L_t)^{-1}(\bar L^{2,2}_t+ \bar L^{2,1}_t- \breve L_t),$ and $\lambda_t= \frac{1}{n} \bar L^{2,1}_t$.
\end{Corollary}
\begin{Remark}
\emph{ Corollary~\ref{cor:finite} provides the optimal information flow topology between the leader and followers for the case when such a topology is not pre-specified.}
\end{Remark}
\section{Infinite Horizon}\label{sec:inf}
In this section,  the result  of Theorem~\ref{thm:finite} is extended to the case of infinite horizon.  To this end, it is assumed that  the dynamics of the agents as well as  the cost function are time-homogeneous; therefore,  the subscript $t$ is omitted to simplify the notation. Given $\beta \in (0,1]$, define
\begin{multline}
J_\infty= \mathbb{E} \Bigg[\sum_{t=1}^\infty \beta^{t-1}\Big( (x^0_t)^\intercal Q^0 x^0_t + (u^0_t)^\intercal R^0 u^0_t\\
+\frac{1}{n} \sum_{i=1}^n  (x^i_t)^\intercal Q x^i_t+( x^i_t-x^0_t)^\intercal P (x^i_t - x^0_t)+ (u^i_t)^\intercal R u^i_t\\
+\frac{1}{2n^2} \sum_{i=1}^n   \sum_{j=1}^n  ( x^i_t-x^j_t)^\intercal H (x^i_t - x^j_t) \Big)\Bigg].
\end{multline}
When $\beta <1$, $J_\infty$  is called infinite-horizon discounted cost and when $\beta=1$, it is called infinite horizon undiscounted cost. The following standard assumption is imposed  on the model.
\begin{Assumption}\label{assump:inf}
Let $(\sqrt{\beta} A,\sqrt{\beta}B)$ and $(\sqrt{\beta}\bar A,\sqrt{\beta}\bar B)$  be stabilizable and $(\sqrt{\beta}A, Q^{1/2})$  and $(\sqrt{\beta}\bar A, \bar Q^{1/2})$ be detectable.
\end{Assumption}
Define the following two algebraic Riccati equations:
\begin{multline}\label{eq:R3}
\breve M = - \beta A^\intercal   \breve M  B  \left( B^\intercal  \breve M  B  + \beta^{-1} R  \right)^{-1}  B^\intercal   \breve M A\\
 + \beta A^\intercal  \breve M A + Q+P+H, 
\end{multline}
\begin{multline}\label{eq:R4}
\bar M = - \beta \bar A^\intercal   \bar M \bar B \left( \bar B^\intercal  \bar M \bar  B  + \beta^{-1} \bar R  \right)^{-1}  \bar B^\intercal   \bar M  \bar A\\
 +\beta \bar  A^\intercal  \bar  M \bar A +\bar Q. 
\end{multline}
\begin{Theorem}\label{thm:inf}
Let  Assumptions~\ref{assump:noise},~\ref{assump:positive} and~\ref{assump:inf} hold. Then,  the optimal  control strategy is given by:
  \begin{align}\label{optimal_action_leader-inf}
    {u^0_t}&=\bar L^{1,1} x^0_t+  \bar L^{1,2}  \bar x_t,\\
    {u^i_t}&=\breve L x^i_t + \bar L^{2,1} x^0_t+ (\bar L^{2,2}-\breve L)\bar x_t,
  \end{align}
  where the gains $\{ \breve L, \bar L\}$ are obtained by  solving  algebraic Riccati equations~\eqref{eq:R3} and~\eqref{eq:R4}  as follows: 
\begin{align}
\breve L&=-\left( B^\intercal  \breve M  B  +  \beta^{-1}R \right)^{-1}  B^\intercal   \breve M  A, \nonumber\\
\bar L
&=\left[ 
\begin{array}{cc}
\bar L^{1,1} & \bar L^{1,2}\\
\bar L^{2,1} & \bar L^{2,2}
\end{array}
\right]=-\left( \bar B^\intercal  \bar M  \bar B + \beta^{-1} \bar R  \right)^{-1}  \bar B^\intercal   \bar M  \bar A.
\end{align}
\end{Theorem}
\begin{proof}
By a simple change of variables,  an  infinite-horizon discounted cost  problem with the 4-tuple
$(A,B,Q,R)$ and  discount factor $\beta$  can be transformed to an infinite-horizon undiscounted cost problem with 4-tuple $(\sqrt{\beta} A, \sqrt{\beta} B, Q, R)$. By applying the same isomorphic transformation as in the proof of Theorem~\ref{thm:finite}  on  the resultant  undiscounted formulation, and using  a similar argument as in that proof,  the $n+1$ decoupled systems are obtained.
\end{proof}

\section{Numerical Example}\label{sec:num}
\textbf{Example 1.} In this section, we present an example of a multi-agent system with a leader and  $100$ followers to verify our theoretical results. Let  the initial state of the leader be $30$, i.e., $x^0_1=30$, and the initial states of the followers be  uniformly distributed random variables in the interval $[0, 20]$. Let also the dynamics of the agents be driven by~\eqref{eq:dynamics-leader} and~\eqref{eq:dynamics-followers} with the following scalar parameters:
$A^0_t=1,  B^0_t=0.3,   A_t=1,   B_t=0.2,  
D^0_t=0.05, D_t=0.01,  E_t=0.01,$
and noises:
\begin{equation}
w^0_t \sim \mathcal{N}(0,0.1), \quad w^i_t \sim \mathcal{N}(0,0.2), \quad \forall i \in \mathbb{N}_n.
\end{equation}
Consider the cost function~\eqref{eq:cost-function} with the following parameters:
\begin{equation}
 Q^0_t=1,  R^0_t=100, Q_t=0.1,   P_t=50,  R_t=50,  H_t=1.
\end{equation}
Assume first that $T=80$ (the finite-horizon case). Using Theorem~\ref{thm:finite},   the optimal trajectories of the leader and followers shown in Figure~\ref{fig} are obtained. Figure~\ref{fig} shows that the states of the followers (thin  colored curves) converge to a small neighborhood of the state  of the leader (thick black curve). The size of this neighborhood depends, in fact, on the noise variance. In the special case when there is no noise, all followers' states approach the state of the leader.

Assume now that $T=\infty$ (the infinite-horizon case with undiscounted cost). Using Theorem~\ref{thm:inf} in this case, the results demonstrated in Figure~\ref{fig2} are obtained, analogously to Figure~\ref{fig}.  Figure~\ref{fig2} shows good convergence results for the followers in the presence of noise.
\begin{figure}[h!]
\hspace{-1cm}
\scalebox{1.2}{
\includegraphics[width=\linewidth]{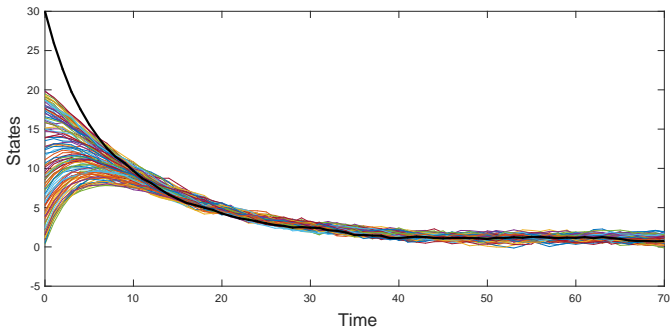} }
\caption{The trajectories of the leader and followers in Example 1 for the finite-horizon case.}\label{fig}
\vspace{-.4cm}
\end{figure}
\begin{figure}[h!]
\hspace{-.6cm}
\scalebox{1.1}{
\includegraphics[width=\linewidth]{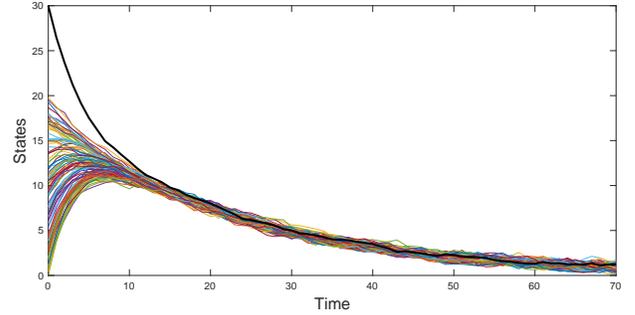} }
\caption{The trajectories of the leader and followers in Example 1 for the infinite-horizon case.}
\label{fig2}
\vspace{-.2cm}
\end{figure}
\section{Conclusions}\label{sec:conc}
In this paper,  optimal  distributed control  of a multi-agent system with a leader-follower structure is investigated. It is assumed that the average of the states of the followers  and the local state of the leader are available to every agent.  The optimal solution is obtained  by solving two decoupled Riccati equations whose computational complexities are independent of the number of followers. In the infinite-horizon case, the Riccati difference equations become algebraic Riccati equations. As an interesting future work, one can consider the case where  the number of followers is sufficiently  large.  In this case,  the average of the states of followers  can be efficiently approximated by the law of large numbers, which means that  the only information to  be shared is the local state of the leader. This implies that  the  communication topology of the network is described by  a directed graph with paths from  the leader to  followers.

\bibliographystyle{IEEEtran}

\bibliography{MFT_Ref} 

\begin{thebibliography}{10}
\providecommand{\url}[1]{#1}
\csname url@samestyle\endcsname
\providecommand{\newblock}{\relax}
\providecommand{\bibinfo}[2]{#2}
\providecommand{\BIBentrySTDinterwordspacing}{\spaceskip=0pt\relax}
\providecommand{\BIBentryALTinterwordstretchfactor}{4}
\providecommand{\BIBentryALTinterwordspacing}{\spaceskip=\fontdimen2\font plus
\BIBentryALTinterwordstretchfactor\fontdimen3\font minus
  \fontdimen4\font\relax}
\providecommand{\BIBforeignlanguage}[2]{{%
\expandafter\ifx\csname l@#1\endcsname\relax
\typeout{** WARNING: IEEEtran.bst: No hyphenation pattern has been}%
\typeout{** loaded for the language `#1'. Using the pattern for}%
\typeout{** the default language instead.}%
\else
\language=\csname l@#1\endcsname
\fi
#2}}
\providecommand{\BIBdecl}{\relax}
\BIBdecl

\bibitem{Fax2004}
J.~A. Fax and R.~M. Murray, ``Information flow and cooperative control of
  vehicle formations,'' \emph{IEEE Transactions on Automatic Control}, vol.~49,
  no.~9, pp. 1465--1476, 2004.

\bibitem{rawat2014wireless}
P.~Rawat, K.~D. Singh, H.~Chaouchi, and J.~M. Bonnin, ``Wireless sensor
  networks: a survey on recent developments and potential synergies,''
  \emph{The Journal of supercomputing}, vol.~68, no.~1, pp. 1--48, 2014.

\bibitem{UAVsurvey2016}
L.~Gupta, R.~Jain, and G.~Vaszkun, ``Survey of important issues in uav
  communication networks,'' \emph{IEEE Communications Surveys Tutorials},
  vol.~18, no.~2, pp. 1123--1152, 2016.

\bibitem{Jadbabaie2003}
A.~Jadbabaie, J.~Lin, and A.~S. Morse, ``Coordination of groups of mobile
  autonomous agents using nearest neighbor rules,'' \emph{IEEE Transactions on
  Automatic Control}, vol.~48, no.~6, pp. 988--1001, 2003.

\bibitem{Olfati2007survey}
R.~Olfati-Saber, J.~A. Fax, and R.~M. Murray, ``Consensus and cooperation in
  networked multi-agent systems,'' \emph{Proceedings of the IEEE}, vol.~95,
  no.~1, pp. 215--233, 2007.

\bibitem{Ren2007survey}
W.~Ren, R.~W. Beard, and E.~M. Atkins, ``Information consensus in multivehicle
  cooperative control,'' \emph{IEEE Control Systems}, vol.~27, no.~2, pp.
  71--82, 2007.

\bibitem{Movric2014}
K.~H. Movric and F.~L. Lewis, ``Cooperative optimal control for multi-agent
  systems on directed graph topologies,'' \emph{IEEE Transactions on Automatic
  Control}, vol.~59, no.~3, pp. 769--774, 2014.

\bibitem{Cao2010}
Y.~Cao and W.~Ren, ``Optimal linear-consensus algorithms: An {LQR}
  perspective,'' \emph{IEEE Transactions on Systems, Man, and Cybernetics, Part
  B (Cybernetics)}, vol.~40, no.~3, pp. 819--830, 2010.

\bibitem{Borelli2008}
F.~Borrelli and T.~Keviczky, ``Distributed {LQR} design for identical
  dynamically decoupled systems,'' \emph{IEEE Transactions on Automatic
  Control}, vol.~53, no.~8, pp. 1901--1912, 2008.

\bibitem{arabneydi2016new}
J.~Arabneydi, ``New concepts in team theory: Mean field teams and reinforcement
  learning,'' Ph.D. dissertation, McGill University, 2016.

\bibitem{ho1972team}
Y.~C. Ho and K.~h. Chu, ``Team decision theory and information structures in
  optimal control problems--part {I},'' \emph{IEEE Transactions on Automatic
  Control}, vol.~17, no.~1, pp. 15--22, 1972.

\bibitem{rotkowitz2006characterization}
M.~Rotkowitz and S.~Lall, ``A characterization of convex problems in
  decentralized control,'' \emph{IEEE Transactions on Automatic Control},
  vol.~51, no.~2, pp. 274--286, 2006.

\bibitem{Witsenhausen1968Counterexample}
H.~Witsenhausen, ``A counterexample in stochastic optimum control,'' \emph{SIAM
  Journal Of Control And Optimization}, vol.~6, pp. 131--147, 1968.

\bibitem{caines1987linear}
P.~E. Caines, \emph{Linear stochastic systems}.\hskip 1em plus 0.5em minus
  0.4em\relax John Wiley and Sons, Inc., 1987.

\end{thebibliography}

\end{document}